\documentclass[a4paper,11pt]{article}

\usepackage[left=2.5cm,right=2.5cm,top=3cm,bottom=3cm,pdftex]{geometry}
\usepackage{amssymb}
\usepackage{amsmath}
\usepackage{amsthm}
\usepackage{dsfont}
\usepackage{url}
\usepackage{natbib}

\usepackage[utf8]{inputenc}
\usepackage[T1]{fontenc}

\usepackage[pdftex]{graphicx}
\DeclareGraphicsExtensions{.png, .pdf, .jpg} % {.jpg,.pdf,.mps,.png,.gif}

\usepackage[pdftex, colorlinks, linkcolor=blue, urlcolor=blue, citecolor=blue, breaklinks=true]{hyperref}

\newtheorem{thm}{Theorem}

\begin{document}

\begin{center}
\Large\textbf{Optimal Designs for Prediction in Two Treatment Groups Random Coefficient Regression Models
} \\[11pt]
\normalsize
Maryna Prus\footnote{Maryna Prus: \href{mailto:maryna.prus@ovgu.de}{maryna.prus@ovgu.de}}\\[11pt]

\footnotesize
Otto-von-Guericke University Magdeburg, Institute for Mathematical Stochastics, 
\\
PF 4120, D-39016 Magdeburg, Germany\\
\normalsize
\end{center}

\begin{quote}
\textbf{Abstract:} The subject of this work is two treatment groups random coefficient regression models, in which observational units receive some group-specific treatments. We provide A- and D-optimal designs (optimal group sizes) for the estimation of fixed effects and the prediction of random effects. We illustrate the obtained results by a numerical example.

\textbf{Keywords:}  Mixed models, estimation and prediction, optimal design, cluster randomization

\end{quote}

\section{Introduction}

The subject of this paper is optimal designs in two treatment groups random coefficient regression (RCR) models, in which observational units receive some group-specific kinds of treatment. These models are typically used for cluster randomized trials. For some real data examples see e.\,g. \cite{pie}. 

Optimal designs for fixed effects models with multiple groups are well discussed in the literature (see e.\,g. \cite{bai}, ch.~3). In models with random coefficients, the estimation of population parameters (fixed effects) is usually of prior interest (see e.\,g \cite{fedo}, \cite{kun}, \cite{can1}). Optimal designs for the prediction of random effects in models with known population parameters have been considered in detail in \cite{gla}. \cite{pru1} provide analytical results for the models with unknown population mean under the assumption of the same design for all individuals.
%\cite{can} investigated the particular case of polynomial growth curves. 
Multiple group models with fixed group sizes were briefly discussed in \cite{pru3}, ch.~6. 

Here, we consider two groups models with unknown population parameters and group specific designs. We provide A- and D-optimality criteria for the estimation and the prediction of fixed and random effects, respectively. Our main focus is optimal designs for the prediction.

The paper is structured in the following way: In Section~2 the two groups RCR model will be introduced. Section~3 presents the best linear unbiased estimator for the population parameter and the best linear unbiased predictor for individual random effects. Section~4 provides analytical results for the designs, which are optimal for the estimation or for the prediction. The paper will be concluded by a short discussion in Section~5.

\section{Two Treatment Groups RCR Model}\label{k2}

In this work we consider RCR models with two treatment groups $G_1$ and $G_2$, where observational units (people, plots, studies, etc.) receive group-specific kinds of treatment, $T_1$ and $T_2$, respectively. Further we will use the term "individuals" instead of "observational units" for simplicity. The first group includes $n_1$ individuals and the second group $n_2$ individuals. The groups sizes $n_1$ and $n_2$ are to be optimized and the total number of individuals $N=n_1+n_2$ in the experiment is fixed. The $k$-th observation at the $i$-th individual is described for the first group by
\begin{equation}\label{mod1}
{Y}_{1ik}= \mu_{\textit{1}i} + \varepsilon_{1ik},\quad i=1, \dots, n_1,\quad  k=1, \dots, K
\end{equation}
and for the second group by
\begin{equation}\label{mod2}
{Y}_{2ik}=\mu_{\textit{2}i} + \varepsilon_{2ik},\quad i=n_1+1, \dots, N,\quad k=1, \dots, K,
\end{equation}
where $K$ is the number of observations per individual, which is assumed to be the same for both groups, $\varepsilon_{1ik}$ and $\varepsilon_{2ik}$ are the observational errors in the first and the second groups with zero expected value and the variances $\mbox{var}(\varepsilon_{1ik})=\sigma_1^2$ and $\mbox{var}(\varepsilon_{2ik})=\sigma_2^2$, respectively. $\mu_{1i}$ and $\mu_{2i}$ are the individual response parameters. 

As it has been already mentioned above, we optimize the group sizes $n_1$ and $n_2$. Therefore, we define the individual parameters for all individuals for both groups: $\mbox{\boldmath{$\theta $}}_i:=( \mu_{1i}, \mu_{2i})^\top$, $i=1, \dots, N$. The parameters can be interpreted as follows: Let individual $i$ be in the second group. Then the parameter $\mu_{1i}$ describes the response, which would be observed at individual $i$ if the individual had received treatment $T_1$, and $\mu_{2i}$ is the usual response parameter of the individual. The latter parametrization allows to identify the best kind of treatment for each individual (for future treatments), which can be useful in practical situations where only one treatment per individual is possible.  

The individual parameters are assumed to have an unknown mean $\mbox{E}(\mbox{\boldmath{$\theta$}}_i)=( \mu_1, \mu_2)^\top=:\mbox{\boldmath{$\theta$}}_0$ and a covariance matrix $\mbox{Cov}(\mbox{\boldmath{$\theta$}}_i)=\textrm{diag}(\sigma_1^2\,u, \sigma_2^2\,v)$ for given dispersions $u>0$ and $v>0$. All individual parameters $\mbox{\boldmath{$\theta $}}_i$ and all observational errors $\varepsilon_{1i'k}$ and $\varepsilon_{2i''k'}$, $i, i', i''=1, \dots, N$, $k, k'=1, \dots, K$, are assumed to be uncorrelated.

Note that this model is not a particular case of the RCR models considered by \cite{pru1}. In contrast to that paper, here the expected values for the response parameters $\mu_{1i}$ and $\mu_{2i}$ in the first and the second groups are not the same (which is equivalent to different regression functions in the parametrization using $\mbox{\boldmath{$\theta$}}_i$) and group sizes are non-fixed. Therefore, the approach proposed by \cite{pru1} cannot be used.

Further we focus on the following contrasts: the population parameter $\alpha_0=\mu_1-\mu_2$ and the individual random parameters $\alpha_i=\mu_{1i}-\mu_{2i}$, $i=1, \dots, N$. $\alpha_0$ describes the difference between the mean parameters $\mu_1$ and $\mu_2$ in the first and in the second group, respectively, and $\alpha_i$ may be interpreted as the difference for individual $i$ between the present response and the response, which could have been observed if the individual had received another treatment. We search for the designs (group sizes), which are optimal for the estimation of $\alpha_0$ or for the prediction of $\alpha_i$.

\section{Estimation and Prediction}\label{k3}

In this section we concentrate on the estimation of the population parameter $\alpha_0$ and the prediction of the individual parameters $\alpha_i$. We use the standard notation $\bar{Y_{\textit{1}}}=\frac{1}{n_1}\sum_{i=1}^{n_1}\,\frac{1}{K}\sum_{k=1}^{K}Y_{\textit{1}ik}$ and $\bar{Y_{\textit{2}}}=\frac{1}{n_2}\sum_{i=n_1+1}^{N}\,\frac{1}{K}\sum_{k=1}^{K}Y_{\textit{2}ik}$ for the mean response in the first and the second treatment group, respectively, and obtain the following best linear unbiased estimator (BLUE) for $\alpha_0$.

\begin{thm}\label{t1} 
\begin{enumerate}
\item[a)] The BLUE for the population parameter $\alpha_0$ is given by 
\begin{equation}\label{a0}
\hat{\alpha}_0=\bar{Y_{\textit{1}}} - \bar{Y_{\textit{2}}}.
\end{equation}
\item[b)] The variance of the BLUE $\hat{\alpha_0}$ is given by 
\begin{equation}\label{var}
\mathrm{var}\left(\hat{\alpha}_0\right) = \frac{\sigma_1^2(Ku+1)}{Kn_1} + \frac{\sigma_2^2(Kv+1)}{Kn_2}.
\end{equation}
\end{enumerate}
\end{thm}

Further we use the notation $\bar{Y}_{\textit{1}i}=\frac{1}{K}\sum_{k=1}^{K}Y_{\textit{1}ik}$  and $\bar{Y}_{2i}=\frac{1}{K}\sum_{k=1}^{K}Y_{\textit{2}ik}$ for the mean individual response for individuals in the first and in the second treatment group, respectively. We obtain the next result for the best linear unbiased predictor (BLUP) for the individual response parameter $\alpha_i$.
\begin{thm} The BLUP for the individual response parameter $\alpha_i$ is given by
\begin{equation}\label{ai}
\hat{\alpha}_{i}=\left\{\begin{array}{cc} \frac{Ku}{Ku+1}\,\bar{Y}_{\textit{1}i}+\frac{1}{Ku+1}\,\bar{Y_{\textit{1}}}-\bar{Y_{\textit{2}}}, & \text{ind.}\,\,\ \text{''}i\,\text{''}\,\,\, \text{in\, G1} \vspace*{0.15cm} \\ 
\bar{Y_{\textit{1}}}-\frac{Kv}{Kv+1}\,\bar{Y}_{\textit{2}i}-\frac{1}{Kv+1}\,\bar{Y_{\textit{2}}}, & \text{ind.}\,\,\ \text{''}i\,\text{''}\,\,\, \text{in\, G2}\,. \end{array}\right.
\end{equation}
\end{thm}

The next theorem presents the mean squared error (MSE) matrix for the total vector $\hat{\mbox{\boldmath{$\alpha$}}}:=\left(\hat{\alpha}_1, ..., \hat{\alpha}_N\right)^\top$ of all BLUPs $\hat{\alpha}_i$ for all individuals.
\begin{thm}\label{t4} 
The MSE matrix of the vector $\hat{\mbox{\boldmath{$\alpha$}}}$ of individual predictors is given by
\begin{equation}\label{mse}
\mathrm{Cov}\left(\hat{\mbox{\boldmath{$\alpha$}}}-\mbox{\boldmath{$\alpha$}}\right)=\left(\begin{array}{cc} \mathbf{A}_{11} & \mathbf{A}_{12} \\ 
\mathbf{A}_{12}^\top & \mathbf{A}_{22} \end{array}\right)
\end{equation}
for
\begin{equation*}
\mathbf{A}_{11}=\left( \frac{\sigma_1^2}{K(Ku+1)n_1} + \frac{\sigma_2^2(Kv+1)}{Kn_2}\right)\mathbf{1}_{n_1}\mathbf{1}_{n_1}^\top+\sigma_1^2\left(\frac{u}{Ku+1} + v \right)\mathbf{I}_{n_1},
\end{equation*}
where $\mathbf{1}_{m}$ denotes the vector of length $m$ with all entries equal to $1$, $\mathbf{I}_m$ is the $m\times m$ identity matrix and $\otimes$ denotes the Kronecker product,
\begin{equation*}
\mathbf{A}_{12}=\left( \frac{\sigma_1^2}{Kn_1} + \frac{\sigma_2^2}{Kn_2}\right)\mathbf{1}_{n_1}\mathbf{1}_{n_2}^\top 
\end{equation*}
and
\begin{equation*}
\mathbf{A}_{22}=\left(\frac{\sigma_1^2(Ku+1)}{Kn_1} + \frac{\sigma_2^2}{K(Kv+1)n_2}\right)\mathbf{1}_{n_2}\mathbf{1}_{n_2}^\top+\sigma_2^2\left(u + \frac{v}{Kv+1}\right)\mathbf{I}_{n_2}.
\end{equation*}
\end{thm}
Proofs of Theorems~\ref{t1}-\ref{t4} are deferred to Appendix~\ref{A1}.

\section{Experimental Design}\label{k4}

We define the experimental (exact) design for the RCR model with two treatment groups $G_1$ and $G_2$ as follows:
\[ \xi:= \left( \begin{array}{cc} 
    G_1 & G_2 \\ n_1 & n_2
\end{array} \right).\]
For analytical purposes, we generalize this to the definition of an approximate design:
\[ \xi:= \left( \begin{array}{cc} 
    G_1 & G_2 \\ w & 1-w
\end{array} \right),\]
where $w=\frac{n_1}{N}$ and $1-w=\frac{n_2}{N}$ are the allocation rates for the first and the second groups, respectively, and only the condition $0 \leq w \leq 1$ has to be satisfied. Then only the optimal allocation rate $w^*$ to the first group has to be determined for finding an optimal design. 

Further we search for the allocation rates, which minimize variance \eqref{var} of the BLUE $\hat{\alpha}_0$ and MSE matrix \eqref{mse} of the BLUP $\hat{\mbox{\boldmath{$\alpha$}}}$ and concentrate on the A- (average) and D- (determinant) optimality criteria.

\subsection{Optimal designs for estimation of population parameter}

For the estimation of the population parameter $\alpha_0$ both A- and D-criteria may be considered to be equal to variance \eqref{var} of the BLUE $\hat{\alpha}_0$. We rewrite the variance of the estimator in terms of the approximate design and receive the following optimality criterion (neglecting the constant factor $(KN)^{-1}$):
\begin{equation}\label{ade}
\Phi_{\alpha_0}(w)= \frac{\sigma_1^2(Ku+1)}{w} + \frac{\sigma_2^2(Kv+1)}{1-w}.
\end{equation}
Criterion function \eqref{ade} can be minimized directly. The optimal allocation rate for the estimation of the population parameter $\alpha_0$ is is given by
\begin{equation}
w^*_{\alpha_0}=\frac{1}{1+\sqrt{\frac{\sigma_2^2(Kv+1)}{\sigma_1^2(Ku+1)}}}.
\end{equation}
Note that the optimal allocation rate $w^*_{\alpha_0}$ to the first group increases with increasing observational error variance $\sigma_1^2$ and the dispersion $u$ of random effects for the first group and decreases with variance parameters $\sigma_2^2$ and $v$ for the second group. Note also that if the observational error variance is the same for both groups ($\sigma_1^2=\sigma_2^2$), $w^*_{\alpha_0}$ 
is larger than $0.5$ for $u>v$ and smaller than $0.5$ for $u<v$.

\subsection{Optimal designs for prediction of individual response parameters}

We define the $A$-criterion for the prediction of the individual response parameters $\mbox{\boldmath{$\alpha$}}=( \alpha_1, ..., \alpha_N)^\top$ as the trace of MSE matrix \eqref{mse}:
\begin{equation}
\Phi_{A, \alpha}: =  \textrm{tr}\left( \textrm{Cov}\left(\hat{\mbox{\boldmath{$\alpha$}}}-\mbox{\boldmath{$\alpha$}}\right) \right).
\end{equation}
%\[ \text{for}\quad \xi= \left( \begin{array}{cc} 
    %\text{G1} & \text{G2} \\ w & 1-w
%\end{array} \right)\]
We extend this definition for approximate designs and receive the following result (neglecting the constant factor $K^{-1}$).
\begin{thm}
The A-criterion for the prediction of the individual response parameters $\mbox{\boldmath{$\alpha$}}=( \alpha_1, ..., \alpha_N)^\top$ is given by
\begin{eqnarray}\label{ap}
\Phi_{A, \alpha}(w)&=& c_1+\sigma_1^2\left(\frac{Ku+1}{w}+N\hspace{0.01cm}w\left(\frac{Ku}{Ku+1}+Kv\right)\right)\nonumber \\
&& +\ \sigma_2^2\left(\frac{Kv+1}{1-w}+N\hspace{0.05cm}(1-w)\left(\frac{Kv}{Kv+1}+Ku\right)\right),
\end{eqnarray}
where
\begin{equation*}
c_1=\sigma_1^2\left(\frac{1}{Ku+1}-Ku-1\right)+\sigma_2^2\left(\frac{1}{Kv+1}-Kv-1\right).
\end{equation*}
\end{thm}

For this criterion no explicit formulas for optimal allocation rates can be provided. For given dispersion matrix of random effects (given values of $u$ an $v$), the problem of optimal designs can be solved numerically. In this work we are however interested in the behavior of optimal designs with respect to the variance parameters. Therefore, we consider some special cases, which illustrate this behavior.

\textbf{Special case 1}: $\sigma_1^2=\sigma_2^2$ and $u=v$

If the variances  $\sigma_1^2$ and $\sigma_2^2$ of the observational errors as well as the dispersions $u$ and $v$ (and consequently the variances $\sigma_1^2\,u$ and $\sigma_2^2\,v$) of the random effects are the same for both groups, A-criterion \eqref{ap} simplifies to
\begin{equation}
\Phi_{A, \alpha}(w)= c_2+\frac{1}{w}+\frac{1}{1-w}, 
\end{equation}
where 
\begin{equation*}
c_2=\frac{NKu'(Ku'+2)+2}{(Ku'+1)^2}-2
\end{equation*}
for $u'=u=v$ (neglecting the factor $Ku'+1$ and the observational errors variance).
We obtain for this criterion the optimal allocation rate $w^*_{A,\alpha}=0.5$, which is also optimal for estimation in the fixed-effects model ($u=v=0$).

\textbf{Special case 2}: $\sigma_1^2=\sigma_2^2$

If only the variances $\sigma_1^2$ and $\sigma_2^2$ of the observational errors are the same for both groups, the A-criterion for the prediction simplifies to 
\begin{eqnarray}
\Phi_{A, \alpha}(w)&=& c_3+\frac{Ku+1}{w}+N\hspace{0.01cm}w\left(\frac{Ku}{Ku+1}+Kv\right)\nonumber\\
&& +\ \frac{Kv+1}{1-w}+N\hspace{0.05cm}(1-w)\left(\frac{Kv}{Kv+1}+Ku\right),
\end{eqnarray}
where 
\begin{equation*}
c_3=\frac{1}{Ku+1}+\frac{1}{Kv+1}-K(u+v)-2,
\end{equation*}
(neglecting the observational errors variance).
The behavior of the optimal allocation rate will be considered for this case in a numerical example later.

The D-criterion for the prediction of $\mbox{\boldmath{$\alpha$}}=( \alpha_1, ..., \alpha_N)^\top$ can be defined as the logarithm of the determinant of MSE matrix \eqref{mse}:
\begin{equation}
\Phi_{D, \alpha}: =  \textrm{log}\,\textrm{det}\left( \textrm{Cov}\left(\hat{\mbox{\boldmath{$\alpha$}}}-\mbox{\boldmath{$\alpha$}}\right) \right).
\end{equation}
%\[ \text{for}\quad \xi= \left( \begin{array}{cc} 
    %\text{G1} & \text{G2} \\ w & 1-w
%\end{array} \right)\]
For approximate designs we obtain the next result.
\begin{thm}
The D-criterion for the prediction of the individual response parameters $\mbox{\boldmath{$\alpha$}}=( \alpha_1, ..., \alpha_N)^\top$ is given by
\begin{equation}\label{d}
\Phi_{D, \alpha}(w)= b_1+w\,N\,\mathrm{log}\left(\frac{\sigma_1^2(Kv+1)}{\sigma_2^2(Ku+1)}\right) + \mathrm{log}\left(\frac{\sigma^2_1(1-w)+\sigma^2_2\,w}{w(1-w)}\right),  
\end{equation}
where
%\begin{equation*}
%t_1=\frac{\sigma_1^2}{Ku+1},
%\end{equation*}
%\begin{equation*}
%t_2=\frac{\sigma_2^2}{Kv+1}.
%\end{equation*}
%and
%\begin{equation*}
%b_1=\mathrm{log}\left(\frac{t_2^{N-1}(v+u(Kv+1))^{N-2}(K(u+v)+1)}{t_1K^2N}\right).
%\end{equation*}
\begin{equation*}
b_1=\mathrm{log}\left(\frac{(\sigma_2^2)^{N-1}(v+u(Kv+1))^{N-2}(Ku+1)(K(u+v)+1)}{\sigma_1^2K^2N(Kv+1)^{N-1}}\right).
\end{equation*}
\end{thm}
\begin{proof}
We compute the determinant of MSE matrix \eqref{mse} using the formula for block-matrices
\begin{equation*}
\mathrm{det}\left(\mathrm{Cov}\left(\hat{\mbox{\boldmath{$\alpha$}}}-\mbox{\boldmath{$\alpha$}}\right)\right)=\mathrm{det}\left(\mathbf{A}_{11}\right)\mathrm{det}\left(\mathbf{A}_{22}-\mathbf{A}_{12}^\top\mathbf{A}_{11}^{-1}\mathbf{A}_{12}\right).
\end{equation*}
Then we rewrite the result in terms of the approximate design and receive criterion \eqref{d}.
\end{proof}
Also for this criterion no finite analytical solutions for optimal designs can be provided. We consider the same special cases as for the A-criterion.

\textbf{Special case 1}: $\sigma_1^2=\sigma_2^2$ and $u=v$

If the variances of the observational errors and the variances of the random effects are the same for the first and the second treatment groups, the D-criterion for the prediction is given by
\begin{equation}
\Phi_{D, \alpha}(w)= b_2-\textrm{log}\left(w(1-w)\right),
\end{equation}
where $b_2=b_1+\textrm{log}(\sigma^2)$ for $\sigma^2=\sigma_1^2=\sigma_2^2$. Then we obtain the optimal allocation rate $w^*_{D,\alpha}=0.5=w^*_{A,\alpha}$, which is also optimal for the fixed-effects model. 
%\begin{fol}
%If the variances of the observational errors as well as the dispersions of the random effects are equal for both groups, the A- and D-optimal designs in the fixed-effects model are A- and D-optimal for the prediction of the individual response parameters in the two groups RCR model.
%\end{fol}

\textbf{Special case 2}: $\sigma_1^2=\sigma_2^2$

If the variances of the observational errors are the same for both groups and the dispersions $u$ and $v$ of random effects may be different, we receive the following D-criterion for the prediction:
\begin{equation}
\Phi_{D, \alpha}(w)= b_2+w\,N\,\mathrm{log}\left(\frac{Kv+1}{Ku+1}\right) - \mathrm{log}\left(w(1-w)\right). 
\end{equation}
If we additionally assume different dispersions of random effects ($u\neq v$), we obtain the next result for the optimal designs.
\begin{thm}
If the variances of the observational errors are the same and the dispersions of the random effects are different for the first and the second treatment groups, the D-optimal allocation rate for the prediction of the individual response parameters $\mbox{\boldmath{$\alpha$}}=( \alpha_1, ..., \alpha_N)^\top$ is given by
\begin{equation}\label{wd}
w^*_{D,\alpha}=  \frac{1}{2a}\left(a+2-\sqrt{a^2+4}\right),
\end{equation}
where 
\begin{equation*}
a=N\,\mathrm{log}\left(\frac{Kv+1}{Ku+1}\right).
\end{equation*}
\end{thm}

Note that the optimal allocation rate $w^*_{D,\alpha}$ to the first group increases with $u$ and decreases with $v$. It can be easily proved that $w^*_{D,\alpha}$ is larger than $0.5$ if $u>v$ and smaller than $0.5$ if $u<v$.

For further considerations we rewrite the optimal allocation rate \eqref{wd} as a function of the ratio $q=\frac{u}{v}$ of the variances of random effects in the first and the second groups and the variance parameter $u$:
\begin{equation*}
a=N\,\mathrm{log}\left(\frac{Ku/q+1}{Ku+1}\right).
\end{equation*}
Than it is easy to verify that  $w^*_{D,\alpha}$ increases with $u$ for $q>1$ ($u>v$) and decreases for $q<1$. 

%We cannot prove this result for the A-criterion analytically. However, we will observe a similar behavior in a numerical example later.

\subsection{Numerical example}

In this section we illustrate the obtained results for the prediction of the individual response parameters by a numerical example. We consider the two groups RCR model with $N=60$ individuals, $K=5$ observations per individual and the same variance of observational errors for both treatment groups: $\sigma_1^2=\sigma_2^2$ (special case~2). We fix the ratio $q=\frac{u}{v}$ of the variances of random effects in the first and the second groups by $q=3$, $q=1$ and $q=0.3$. Figures~1 and 2 illustrate the behavior of the optimal allocation rates for the A- and D-criteria in dependence of the rescaled random effects variance in the first group $\rho={u}/{(1+u)}$, which is monotonic in $u$ and has been used instead the of random effects variance itself to cover all values of the variance by the finite interval $[0,1]$.
\begin{figure}[ht]
    \begin{minipage}[]{8.2 cm}
       \centering
       \includegraphics[width=82mm]{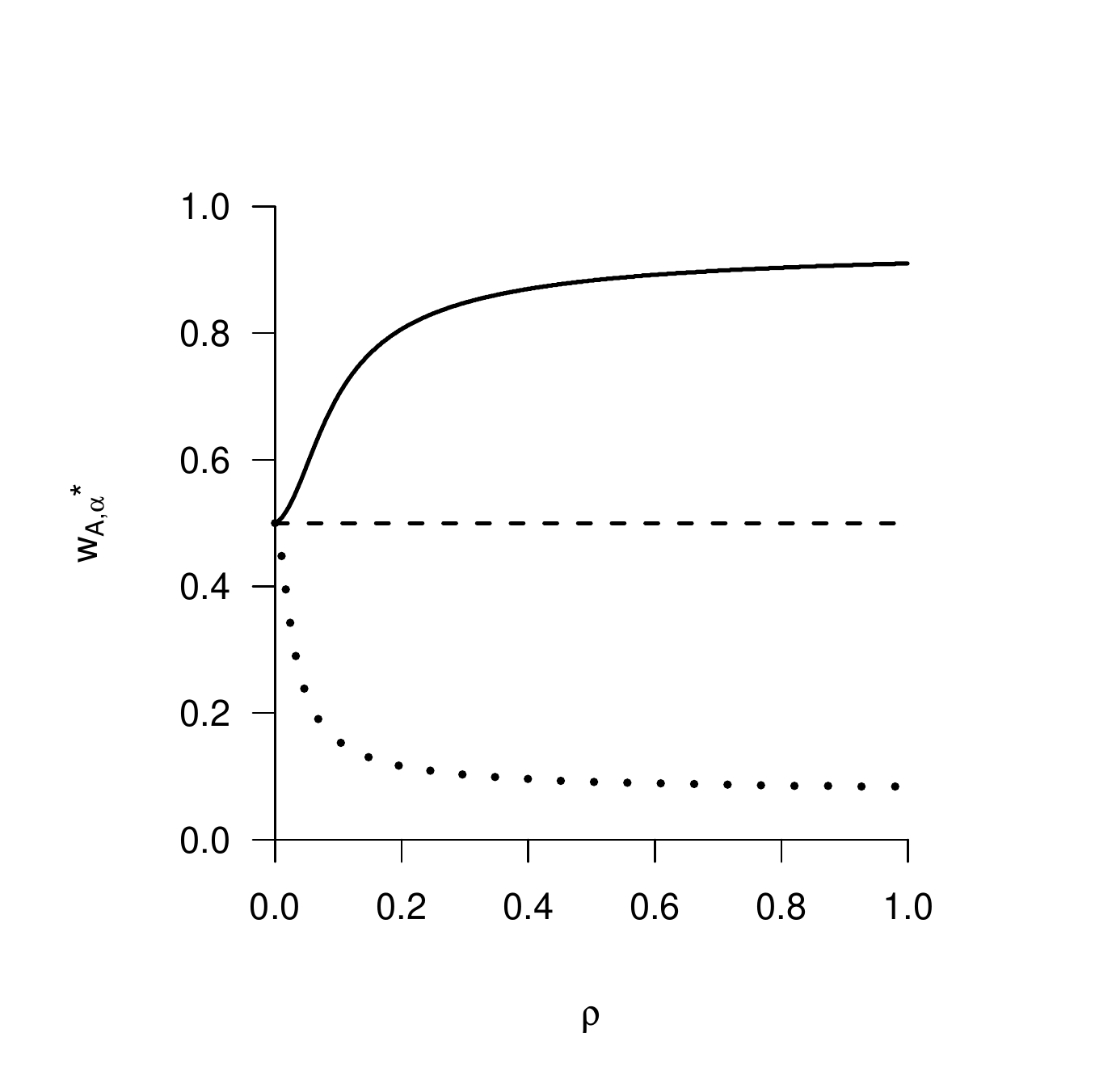}
       \label{a1}
       \end{minipage}
       \begin{minipage}[]{8.2 cm}
       \centering
       \includegraphics[width=82mm]{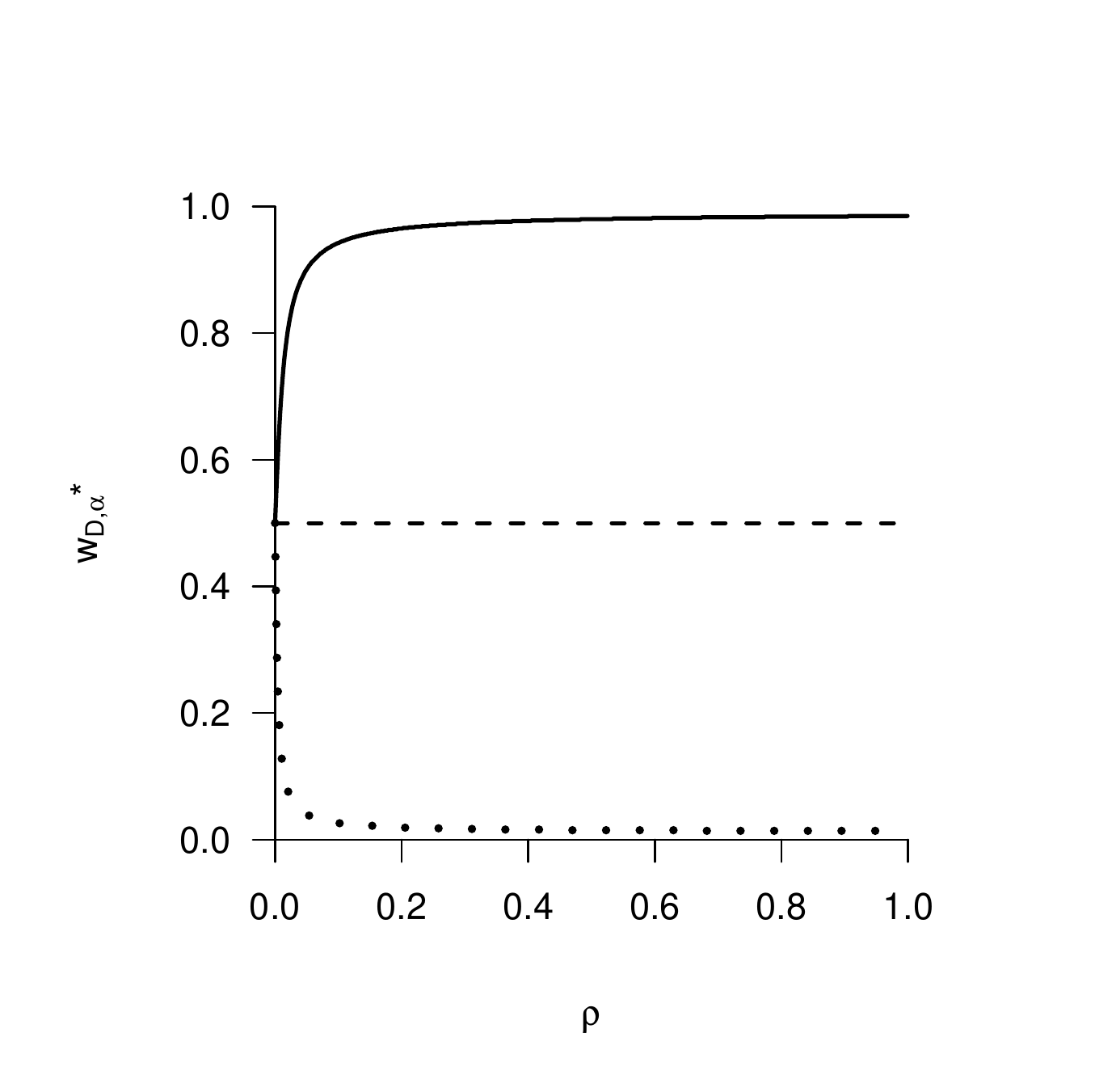}
       \label{a2}
       \end{minipage}
       \vspace{-1mm}
       \hspace*{5 mm}
       \begin{minipage}[]{6.5 cm}
       Figure 1: A-optimal allocation rate $w^*$ for variance ratios $q=3$ (solid line), $q=1$ (dashed line) and $q=0.3$ (dotted line)
       \end{minipage}
       \hspace{15 mm}
       \begin{minipage}[]{6.5 cm}
       Figure 2: D-optimal allocation rate $w^*$ for variance ratios $q=3$ (solid line), $q=1$ (dashed line) and $q=0.3$ (dotted line) 
       \end{minipage}
    \end{figure} 
		
As we can observe on the graphics, the optimal allocation rate to the first group increases with the rescaled variance $\rho$ from $0.5$ for $\rho\to 0$ to $0.910$ for the A-criterion and to $0.985$ for the D-criterion for $\rho\to\infty$ if $q=3$. If $q=0.3$, the optimal allocation rate decreases from $0.5$ to $0.083$ and $0.014$ for the A- and D-criterion, respectively. For $q=1$ the model coincides with that considered in special case~1 and the optimal design remains the same ($w^*_{A,\alpha}$=$w^*_{D,\alpha}$=0.5) for all values of $u$.

Figures~3 and 4 exhibit the efficiencies of the balanced design $w=0.5$ for the prediction in the two groups model for the A- and D-criteria. For computing the A- and D-efficiencies, we use the formulas
\begin{equation}
\hspace{1.2cm}\mathrm{eff}_A=\frac{\Phi_{A, \alpha}({{w}^*_{A,\alpha}})}{\Phi_{A, \alpha}(0.5)}
\end{equation}
and
\begin{equation}
\hspace{0.8cm}\mathrm{eff}_D=\left(\frac{\mathrm{exp}(\Phi_{D, \alpha}({{w}^*_{D,\alpha}}))}{\mathrm{exp}(\Phi_{D, \alpha}(0.5)}\right)^{\frac{1}{N}},
\end{equation}
respectively. 
\begin{figure}[ht]
    \begin{minipage}[]{8.2 cm}
       \centering
       \includegraphics[width=82mm]{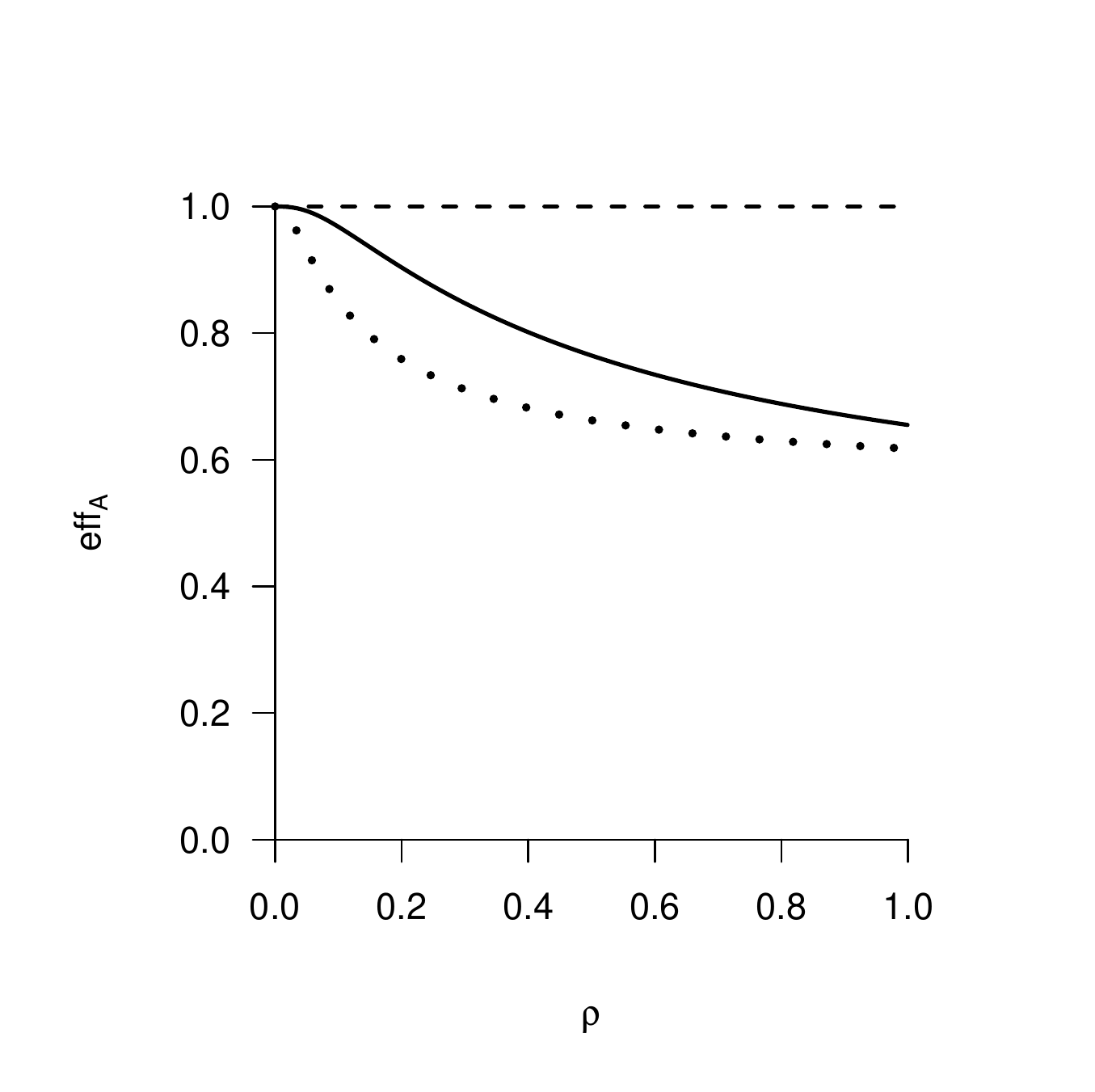}
       \label{a3}
       \end{minipage}
       \begin{minipage}[]{8.2 cm}
       \centering
       \includegraphics[width=82mm]{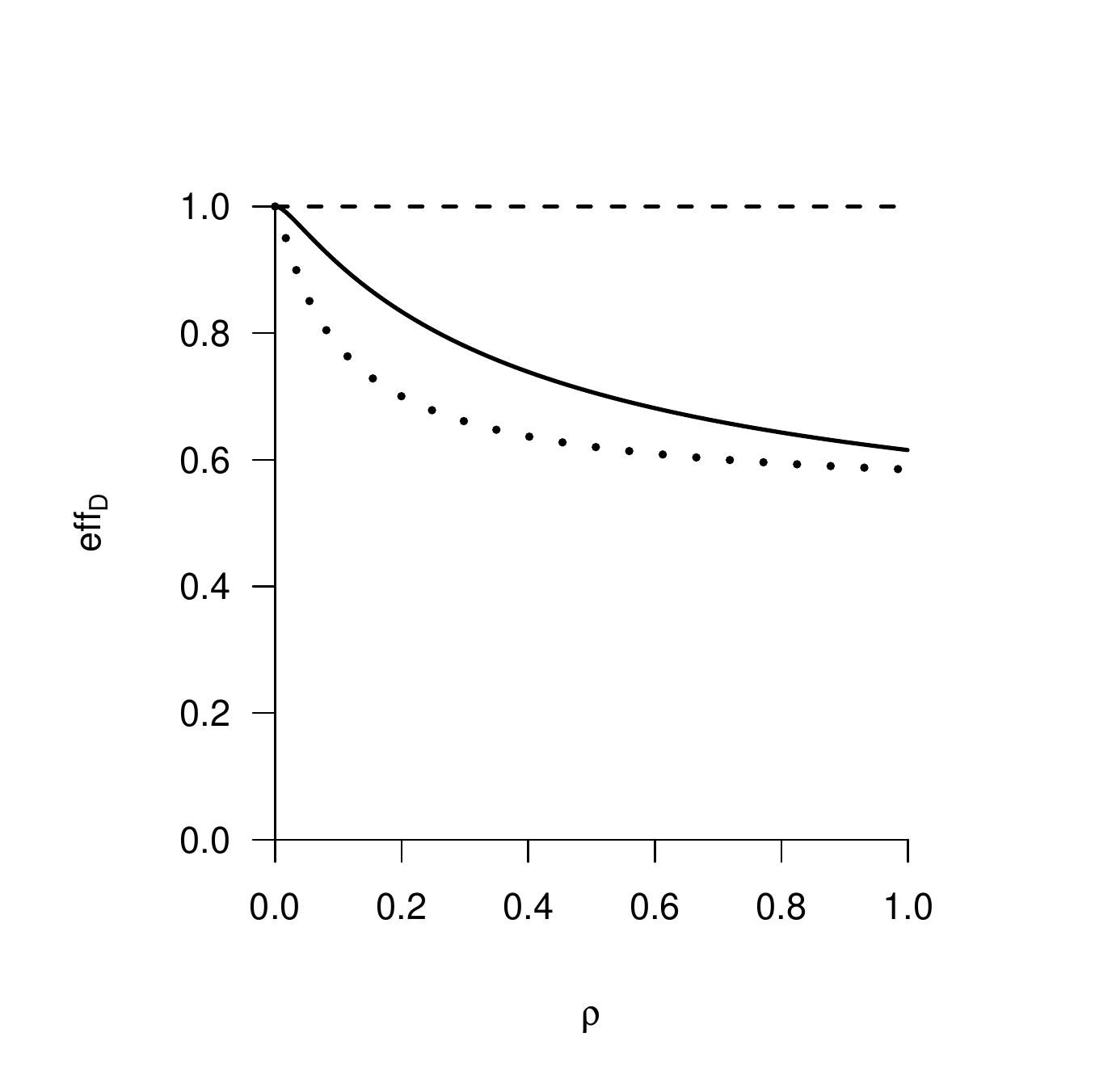}
       \label{a4}
       \end{minipage}
       \vspace{-1mm}
       \hspace*{5 mm}
       \begin{minipage}[]{6.5 cm}
       Figure 3: A-efficiency of the balanced design $w=0.5$ for variance ratios $q=3$ (solid line), $q=1$ (dashed line) and $q=0.3$ (dotted line) 
       \end{minipage}
       \hspace{15 mm}
       \begin{minipage}[]{6.5 cm}
       Figure 4: D-efficiency of the balanced design $w=0.5$ for variance ratios $q=3$ (solid line), $q=1$ (dashed line) and $q=0.3$ (dotted line)
       \end{minipage}
    \end{figure} 
		
As we can observe, the efficiency of the balanced design decreases with increasing values of $\rho$ from $1$ for $\rho\to 0$ to $0.655$ and $0.615$ if $q=3$ and to $0.618$ and $0.585$ if $q=0.3$ for the A- and D-criteria, respectively. For $q=1$ the balanced design is optimal for the prediction, which explains the efficiency equal to $1$ for all values of the variance.

\section{Discussion}
In this work we have considered RCR models with two treatment groups. We have obtained the A- and D-optimality criteria for the estimation of the population parameter and the prediction of the individual response. For a particular case of the same observational error variance for both groups, we illustrate the behavior of the optimal designs by a numerical example. The optimal allocation rate to the first treatment group turns out to be larger than $0.5$ if the variance of individual random effects in the first group is larger than in the second group. Otherwise, the optimal allocation rate is smaller than $0.5$. The efficiency of the balanced design, which assigns equal group sizes, is relatively high only for small values of the variances of random effects. The efficiency decreases fast with increasing variance.

For simplicity, we have assumed a diagonal covariance matrix of random effects. For more general covariance structure further considerations are needed. We have also assumed the same number of observations for all individuals. Optimal designs for models with different numbers of observations for different individuals may be one of the next steps in the research. Moreover, optimal designs for RCR models with more than two groups can be investigated in the future. Furthermore, some research on more robust design criteria (for example, minimax or maximin efficiency), which are not sensible with respect to variance parameters, may be an interesting extension of this work.

\appendix

%\section{Appendix}

\section{Proofs of Theorems~1-4}\label{A1}

The two treatment groups RCR model described by formulas \eqref{mod1} and \eqref{mod2} may be recognized as a special case of the general linear mixed model 
\begin{equation}\label{mod}
\mathbf{Y}=\mathbf{X} \mbox{\boldmath{$\beta $}} + \mathbf{Z} \mbox{\boldmath{$\gamma$}} + \mbox{\boldmath{$\varepsilon$}}
\end{equation}
with specific design matrices $\mathbf{X}$ and $\mathbf{Z}$ for fixed and random effects, respectively. $\mbox{\boldmath{$\varepsilon$}}$ are the observational errors, $\mbox{\boldmath{$\beta $}}$ denotes the fixed effects vector and $\mbox{\boldmath{$\gamma $}}$ are the random effects. The random effects and the observational errors are assumed to have zero mean and to be all uncorrelated with corresponding full rank covariance matrices $\mbox{Cov}\,(\mbox{\boldmath{$\gamma $}})=\mathbf{G}$ and $\mbox{Cov}\,(\mbox{\boldmath{$\varepsilon $}})=\mathbf{R}$. 

In model \eqref{mod} the BLUE for $\mbox{\boldmath{$\beta $}}$ and the BLUP for $\mbox{\boldmath{$\gamma $}}$ are solutions of the mixed model equations 
\begin{equation}\label{blup} 
\left( \begin{array}{c} \hat{\mbox{\boldmath{$\beta$}}} \\ \hat{\mbox{\boldmath{$\gamma$}}}
       \end{array} \right)
= {\left( \begin{array}{cc} \mathbf{X}^\top \mathbf{R}^{-1}\mathbf{X} & \mathbf{X}^\top \mathbf{R}^{-1}\mathbf{Z} \\ \mathbf{Z}^\top \mathbf{R}^{-1}\mathbf{X} & \mathbf{Z}^\top \mathbf{R}^{-1}\mathbf{Z}+\mathbf{G}^{-1}
          \end{array} \right)^{-1}}  
\left( \begin{array}{c} \mathbf{X}^\top \mathbf{R}^{-1}\mathbf{Y} \\ \mathbf{Z}^\top \mathbf{R}^{-1}\mathbf{Y} \end{array} \right) 
\end{equation}
if the fixed effects design matrix $\mathbf{X}$ has full column rank (see e.\,g. \citet{hen1} and \citet{chr}). According to \citet{hen3}, the joint MSE matrix for both $\hat{\mbox{\boldmath{$\beta $}}}$ and $\hat{\mbox{\boldmath{$\gamma $}}}$ is given by
\begin{equation}\label{cov1}
\mbox{Cov}\,\left( \begin{array}{c} \hat{\mbox{\boldmath{$\beta $}}} \\ \hat{\mbox{\boldmath{$\gamma $}}}-\mbox{\boldmath{$\gamma $}} 
                   \end{array} \right)
={\left( \begin{array}{cc} \mathbf{X}^\top \mathbf{R}^{-1}\mathbf{X} & \mathbf{X}^\top \mathbf{R}^{-1}\mathbf{Z} \\ \mathbf{Z}^\top \mathbf{R}^{-1}\mathbf{X} & \mathbf{Z}^\top \mathbf{R}^{-1}\mathbf{Z}+\mathbf{G}^{-1}
          \end{array} \right)^{-1}} .
\end{equation} 

To make use of the theoretical results available for the general linear mixed model, we rewrite the two groups RCR model in form \eqref{mod}:
\begin{equation}\label{mod3}
\mathbf{Y}=
\left(\begin{array}{c} \mathbf{1}_{Kn_1}e_1^\top \\ \mathbf{1}_{Kn_2}e_2^\top \end{array}\right)\mbox{\boldmath{$\beta $}}
+ \left(\begin{array}{cc} \mathbf{I}_{n_1}\otimes\left(\mathbf{1}_{K}e_1^\top\right) & \mathbf{0} \\ \mathbf{0} & \mathbf{I}_{n_2}\otimes\left(\mathbf{1}_{K}e_2^\top\right) \end{array}\right)\mbox{\boldmath{$\gamma $}}
+\mbox{\boldmath{$\varepsilon$}},
\end{equation}
where $\mbox{\boldmath{$\beta $}}=\mbox{\boldmath{$\theta $}}_0$, $\mbox{\boldmath{$\gamma $}}=\mbox{\boldmath{$\theta $}}-\left(\mathbf{1}_N\otimes \mathbf{I}_2 \right)\mbox{\boldmath{$\beta $}}$, $\mbox{\boldmath{$\theta $}}=(\theta_1, \dots, \theta_N)$ and  $e_m$ denotes the $m$-th unit vector. The covariance matrices of the random effects and the observational errors in model \eqref{mod3}  are given by $\mathbf{G}=\mathbf{I}_N\otimes\textrm{diag}(\sigma_1^2\,u, \sigma_2^2\,v)$ and $\mathbf{R}=\textrm{block-diag}(\sigma_1^2\,\mathbf{I}_{Kn_1}, \sigma_2^2\,\mathbf{I}_{Kn_2})$, respectively.

Using formula \eqref{blup} we obtain the BLUEs $\hat{\mu}_1= \bar{Y_{\textit{1}}}$ and $\hat{\mu}_2=\bar{Y_{\textit{2}}}$ for the fixed effects and the BLUPs 
\begin{equation}
\hat{\mu}_{1i}=\left\{\begin{array}{cc} \frac{Ku}{Ku+1}\,\bar{Y}_{\textit{1}i}+\frac{1}{Ku+1}\,\bar{Y_{\textit{1}}}, & \text{ind.}\,\,\ \text{''}i\,\text{''}\,\,\, \text{in\, $G_1$} \vspace*{0.15cm} \\ 
\bar{Y_{\textit{1}}}, & \text{ind.}\,\,\ \text{''}i\,\text{''}\,\,\, \text{in\, $G_2$} \end{array}\right.
\end{equation}
and
\begin{equation}
\hat{\mu}_{2i}=\left\{\begin{array}{cc} \frac{Kv}{Kv+1}\,\bar{Y}_{\textit{2}i}+\frac{1}{Kv+1}\,\bar{Y_{\textit{2}}}, & \text{ind.}\,\,\ \text{''}i\,\text{''}\,\,\, \text{in\, $G_2$} \vspace*{0.15cm} \\ 
\bar{Y_{\textit{2}}}, & \text{ind.}\,\,\ \text{''}i\,\text{''}\,\,\, \text{in\, $G_1$} \end{array}\right.
\end{equation}
for the random effects. Then the BLUE and the BLUP for the contrasts $\alpha_0$ and $\alpha_i$ can be computed as $\hat{\alpha}_0=\hat{\mu}_1 - \hat{\mu}_2$ and $\hat{\alpha}_i=\hat{\mu}_{1i} - \hat{\mu}_{2i}$ and result to formulas \eqref{a0} and \eqref{ai}, respectively. Variance \eqref{var} of the estimator $\hat{\alpha}_0$ can be determined directly.

Using formula \eqref{cov1} we obtain the following joint MSE matrix for both $\hat{\mbox{\boldmath{$\beta $}}}$ and $\hat{\mbox{\boldmath{$\gamma $}}}$:
\begin{equation}\label{cov2}
\mbox{Cov}\,\left( \begin{array}{c} \hat{\mbox{\boldmath{$\beta $}}} \\ \hat{\mbox{\boldmath{$\gamma $}}}-\mbox{\boldmath{$\gamma $}} 
                   \end{array} \right)
=\left(
\begin{array}{cc} \mathbf{C}_{11} & \mathbf{C}_{12} \\ \mathbf{C}_{12}^\top  & \mathbf{C}_{22} \end{array} \right),
\end{equation}
where 
\begin{equation*}
\mathbf{C}_{11}=\left(\begin{array}{cc}\frac{\sigma_1^2(Ku+1)}{Kn_1} & 0 \\ 0 & \frac{\sigma_2^2(Kv+1)}{Kn_2}\end{array}\right),
\end{equation*}
\begin{equation*}
\mathbf{C}_{12}=-\left(\begin{array}{cc}\frac{1}{n_1}\,\sigma_1^2\,u\,\mathbf{1}_{n_1}^\top\otimes e_1^\top & \mathbf{0} \\ \mathbf{0} & \frac{1}{n_2}\,\sigma_2^2\,v\,\mathbf{1}_{n_2}^\top\otimes e_2^\top\end{array}\right)
\end{equation*}
and
\begin{equation*}
\mathbf{C}_{22}= \left(\begin{array}{cc} \mathbf{B}_1 & \mathbf{0} \\ \mathbf{0} & \mathbf{B}_2 \end{array}\right)
\end{equation*}
for
\begin{equation*}
\mathbf{B}_1= \sigma_1^2\left(\frac{Ku^2}{n_1(Ku+1)}\mathbf{1}_{n_1}\mathbf{1}_{n_1}^\top\otimes(e_1e_1^\top)+\mathbf{I}_{n_1}\otimes\textrm{diag}\left(\frac{u}{Ku+1}, v\right)\right)
\end{equation*}
and
\begin{equation*}
\mathbf{B}_2= \sigma_2^2\left(\frac{Kv^2}{n_2(Kv+1)}\mathbf{1}_{n_2}\mathbf{1}_{n_2}^\top\otimes(e_2e_2^\top)+\mathbf{I}_{n_2}\otimes\textrm{diag}\left(u, \frac{v}{Kv+1}\right)\right).
\end{equation*}
%As we can see by formula \eqref{cov2}, the covariance matrix of $\hat{\mbox{\boldmath{$\beta $}}}$ is equal to $\mathbf{C}_{11}$. Then we present the variance of the estimator $\hat{\alpha}_0$ in form 
%\begin{equation*}\mathrm{var}(\hat{\alpha}_0)=\left(1,-1\right)\mathrm{Cov}(\hat{\mbox{\boldmath{$\beta $}}})\left(1,-1\right)^\top
%\end{equation*} 
%and obtain result \eqref{var} of Theorem~\ref{t1}.
The MSE matrix of the prediction $\hat{\mbox{\boldmath{$\theta$}}}$ can be written in terms of joint MSE matrix \eqref{cov2}:
\begin{equation}
\textrm{Cov}\left(\hat{\mbox{\boldmath{$\theta$}}}-\mbox{\boldmath{$\theta$}}\right)=\left(\mathbf{1}_N\otimes\mathbf{I}_2\right)\mathbf{C}_{11}\left(\mathbf{1}_N^\top\otimes\mathbf{I}_2\right)+\left(\mathbf{1}_N\otimes\mathbf{I}_2\right)\mathbf{C}_{12}+\mathbf{C}_{12}^\top\left(\mathbf{1}_N^\top\otimes\mathbf{I}_2\right)+\mathbf{C}_{22}.
\end{equation}
Using this formula we obtain 
\begin{equation*}
\textrm{Cov}\left(\hat{\mbox{\boldmath{$\theta$}}}-\mbox{\boldmath{$\theta$}}\right)=\left(\begin{array}{cc} \mathbf{H}_{11} & \mathbf{H}_{12} \\ 
\mathbf{H}_{12}^\top & \mathbf{H}_{22} \end{array}\right),
\end{equation*}
where
\begin{equation*}
\mathbf{H}_{11}=\mathbf{1}_{n_1}\mathbf{1}_{n_1}^\top\otimes\left(\begin{array}{cc} \frac{\sigma_1^2}{K(Ku+1)n_1} & 0 \\ 0 & \frac{\sigma_2^2(Kv+1)}{Kn_2}\end{array}\right)+\sigma_1^2\,\mathbf{I}_{n_1}\otimes\left(\begin{array}{cc} \frac{u}{Ku+1} & 0 \\ 0 & v \end{array}\right),
\end{equation*}
\begin{equation*}
\mathbf{H}_{12}=\mathbf{1}_{n_1}\mathbf{1}_{n_2}^\top\otimes \left(\begin{array}{cc} \frac{\sigma_1^2}{Kn_1} & 0 \\ 0 & \frac{\sigma_2^2}{Kn_2} \end{array}\right)
\end{equation*}
and
\begin{equation*}
\mathbf{H}_{22}=\mathbf{1}_{n_2}\mathbf{1}_{n_2}^\top\otimes\left(\begin{array}{cc} \frac{\sigma_1^2(Ku+1)}{Kn_1} & 0 \\ 0 & \frac{\sigma_2^2}{K(Kv+1)n_2}\end{array}\right)+\sigma_2^2\,\mathbf{I}_{n_2}\otimes\left(\begin{array}{cc} u & 0 \\ 0 & \frac{v}{Kv+1} \end{array}\right).
\end{equation*}
Then we present the MSE matrix of the predictor $\hat{\mbox{\boldmath{$\alpha$}}}$ in form
\begin{equation*}
\textrm{Cov}\left(\hat{\mbox{\boldmath{$\alpha$}}}-\mbox{\boldmath{$\alpha$}}\right)=\left(\mathbf{I}_N\otimes\mathbf{1}_2^\top\right)   \textrm{Cov}\left(\hat{\mbox{\boldmath{$\theta$}}}-\mbox{\boldmath{$\theta$}}\right)\left(\mathbf{I}_N\otimes\mathbf{1}_2\right)
\end{equation*} 
and receive result \eqref{mse} of Theorem~\ref{t4}.

\bibliographystyle{natbib}
\bibliography{prus5}

\end{document}